\documentclass[12pt]{amsart}
\usepackage{enumerate,amsmath,amssymb,bm}
\usepackage{url}
\usepackage{xspace}

\usepackage[margin=1in]{geometry}
\DeclareMathOperator{\pnt}{\raise 0.5mm \hbox{\large\textbf{.}}}

\newcommand{\note}[2][ ]{}%dummy macro

\newtheorem{theorem}{Theorem}%[section]
\newtheorem{lemma}[theorem]{Lemma}

\newtheorem{corollary}[theorem]{Corollary}
\newtheorem{conjecture}[theorem]{Conjecture}

\newtheorem{problem}[theorem]{Problem}

\theoremstyle{definition}
\newtheorem{remark}[theorem]{Remark}

\usepackage[pdfauthor={William J. Keith and Fabrizio Zanello},
            pdftitle={Parity of the coefficients of certain eta-quotients, II},
            pdfsubject={enumerative combinatorics},
            pdfkeywords={regular partitions, parity},
            pdfproducer={Latex with hyperref},
            pdfcreator={latex->dvips->ps2pdf},
            pdfpagemode=UseNone,
            bookmarksopen=false,
            bookmarksnumbered=true]{hyperref}

\begin{document}
\title[Parity of eta-quotients, II: The case of even-regular partitions]{Parity of the coefficients of certain eta-quotients, II: The case of even-regular partitions}
\author{William J. Keith and Fabrizio Zanello} \address{Department of Mathematical Sciences\\ Michigan Tech\\ Houghton, MI  49931-1295}
\email{wjkeith@mtu.edu, zanello@mtu.edu}
\thanks{2020 {\em Mathematics Subject Classification.} Primary: 11P83; Secondary:  05A17, 11P84, 11P82, 11F33.\\\indent 
{\em Key words and phrases.} Partition function; regular partition; density odd values; partition identity; modular form modulo 2; eta-quotient; parity of the partition function.}

\maketitle
%\linenumbers

\begin{abstract}
We continue our study of the density of the odd values of eta-quotients, here focusing on the $m$-regular partition functions $b_m$ for $m$ even. Based on extensive computational evidence, we propose an elegant conjecture which, in particular, completely classifies such densities: Let $m = 2^j m_0$ with $m_0$ odd. If $2^j < m_0$, then the odd density of $b_m$ is $1/2$; moreover, such density is equal to $1/2$ on \emph{every} (nonconstant) subprogression $An+B$. If $2^j > m_0$, then $b_m$, which is already known to have density zero, is identically even on  infinitely many non-nested subprogressions. This and all other conjectures of this paper are consistent with our \lq \lq master conjecture'' on eta-quotients presented in the previous work.\\ 
In general, our results on $b_m$ for $m$ even determine behaviors considerably different from the case of $m$ odd. Also interesting, it frequently happens that on subprogressions $An+B$, $b_m$ matches the parity of the multipartition functions $p_t$, for certain values of $t$. We make a suitable use of Ramanujan-Kolberg identities to deduce a large class of such results; as an example, $b_{28}(49n+12) \equiv p_3(7n+2) \pmod{2}$. Additional consequences are several \lq \lq almost always congruences'' for various $b_m$, as well as new parity results specifically for $b_{11}$.\\
We wrap up our work with a much simpler proof of the main result of a recent paper by Cherubini-Mercuri, which fully characterized the parity of $b_8$.
\end{abstract}

\section{Introduction and Statements of the Results and Conjectures}

The present paper is the second in a series that addresses the behavior modulo 2 of the $m$-regular partitions $b_m(n)$.  In turn, this is part of a far-reaching project on the longstanding open problem of determining the density of the odd values of the partition function $p(n)$.  The generating functions of both $b_m(n)$ and $p(n)$ are viewed as special cases of a  broad class of eta-quotients defined in our previous work \cite{KZ}, whose parity we believe to be regulated by a \lq \lq master conjecture'' (\cite{KZ}, Conjecture 4). We refer to the next section for all definitions and relevant terminology.

In \cite{KZ}, we mainly explored $b_m$ for $m$ odd, proving congruences or similarities of the form 
$$\sum_{n=0}^\infty b_m(An+B) q^n \equiv 0 \pmod{2}$$
and
$$\sum_{n=0}^\infty b_m(An+B) q^n \equiv \sum_{n=0}^\infty b_m(Cn+D) q^{jn} \pmod{2}.$$
(For the remainder of this paper, all congruences will be assumed to be modulo 2 unless explicitly stated otherwise. Congruences for power series hold in the ring $\mathbb{Z}_2[[q]]$.)

In \cite{KZ}, we avoided working with $m$ even since such identities often did not appear to exist except in trivial instances, a statement which we are now prepared to formalize as a conjecture.  We have at least a partial explanation for why this is the case: further work, presented here, shows that even-regular partitions often match the parity of $t$-multipartitions, for suitable values of $t$.  It was conjectured by Subbarao and proved by Ono \cite{Ono3}, Getz \cite{Getz}, Boylan and Ono \cite{BoylanOno}, and ultimately Radu \cite{RaduNoEvens} that there exist no nonconstant subprogressions $An+B$ where the partition function is identically even, i.e., $p(An+B)$ is never zero mod 2 for all $n$.  In fact,  this (and much more) is believed to hold true for any $t$-multipartition function with $t$ odd (see \cite{JudgeThesis,JZ}), the case $t=1$ being that of $p(n)$.  Heuristically, a similar behavior suggests that any function that matches one of the above, even in part, is unlikely to have identically even subprogressions.

Our first main theorem in this paper is the following:
\begin{theorem}\label{mainthm}
Let $m = 2^jm_0$ be even, with $m_0$ odd satisfying $5 \leq m_0 \leq 23$, $3 \nmid m_0$, and $m_0 > 2^j$. Let $\beta_m = \left(m_0^2 - 1\right)/24$.  Then
$$b_m\left(m_0^2 n - \beta_m + m_0 B\right) \equiv p_{m_0 - 2^j} (m_0 n + B)$$
for all $n$ and for the following $B$:
\begin{center}\begin{tabular}{|c|c|c|c|c|c|}
\hline $m_0$ & $m$ & $B$ & $m_0$ & $m$ & $B$ \\
\hline 5 & 10 & 1, 3 & 17 & 34 & 3, 5, 6, 8, 9, 11, 15, 16\\
\hline & 20 & 1, 2 & & 68 & 1, 5, 6, 10, 12, 13, 15, 16 \\
\hline 7 & 14 & 1, 5, 7 & & 136 &  2, 3, 7, 9, 10, 12, 13, 15 \\
\hline & 28 & 2, 3, 5 & & 272 & 1, 3, 4, 6, 7, 9, 13, 14 \\
\hline 11 & 22 & 1, 5, 6, 7, 9 & 19 & 38 & 1, 3, 8, 9, 12, 15, 16, 17, 18 \\
\hline & 44 & 1, 2, 3, 7, 10 & & 76 & 2, 5, 6, 11, 13, 15, 16, 17, 18 \\
\hline & 88 & 2, 3, 4, 6, 9 & & 152 & 3, 4, 7, 10, 11, 12, 13, 15, 17 \\
\hline 13 & 26 & 3, 6, 7, 8, 9, 12 & & 304 & 1, 3, 5, 6, 7, 8, 11, 14, 15 \\
\hline & 52 & 1, 3, 5, 6, 11, 12 & 23 & 46 & 3, 5, 8, 9, 12, 13, 15, 17, 18, 19, 20 \\
\hline & 104 & 2, 6, 9, 10, 11, 12 & & 92 & 1, 3, 6, 7, 10, 11, 13, 15, 16, 17, 18  \\
\hline & & & & 184 & 2, 3, 6, 7, 9, 11, 12, 13, 14, 20, 22 \\
\hline & & & & 368 & 1, 3, 4, 5, 6, 12, 14, 17, 18, 21, 22 \\
\hline
\end{tabular}
\end{center}
\end{theorem}
{\ }\\
\begin{remark}
What Theorem \ref{mainthm} means is that the coefficients in the power series expansion of
$$q^{\frac{m_0^2-1}{24}} \cdot \frac{f_m}{f_1} + \frac{1}{f_{m_0}^{{m_0}-2^d}}$$
are identically zero mod 2 in the progressions $m_0^2n+m_0B$.  (Here and throughout, we employ the standard notation $f_j = \prod_{i=1}^\infty (1-q^{ji})$.)  For instance, a case of this theorem is that, in
$$q^2 \frac{f_{28}}{f_1} + \frac{1}{f_7^3},$$
the coefficients of $q^{49n+14}$, $q^{49n+21}$, and $q^{49n+35}$ are all even.
\end{remark}

In addition, we have the following results for $m_0 = 25$ and $m_0 = 29$:
\begin{theorem}\label{largermod}
We have:
$$b_{50}(875n - 26 + 25B) \equiv p_{23} (35n+B)$$
for all $n$ when $B \in \{ 11, 18, 21, 26, 28, 33 \};$
$$b_{100}(875n - 26 + 25B) \equiv p_{21}(35n+B)$$
for all $n$ when $B \in \{ 6, 12, 17, 26, 27, 31 \};$
$$b_{200}(6125n - 26 + 25B) \equiv p_{17}(245n+B)$$
for all $n$ when $B \in \{ 22, 29, 92,127, 169, 239 \};$
$$b_{400}(1125 n - 26 + 25 B) \equiv p_{9} (45 n + B)$$
for all $n$ when $B \in \{3, 9, 18, 39\};$ and
$$b_{232}(841 n - 35 + 29 B) \equiv p_{21} (29 n + B)$$
for all $n$ when $B \in \{3, 6, 10, 12, 13, 14, 15, 18, 20, 23, 24, 25, 26, 28 \}.$
\end{theorem}

Note that, here, the subprogressions are sometimes of greater modulus than $m_0^2$.  Additionally, while in the previous theorem all integer values of $j$ in the range $1\leq j \leq \text{log}_2 m_0$ appeared, that is not the case in Theorem \ref{largermod}.  We suspect no further subprogressions of this type exist for $m_0 = 29$. We discuss the reasons for it after the proof of the theorem.

On the other hand, we conjecture that while more elaborate similarities may occur when $3 \vert m_0$, no result exists of the simple form given in Theorem \ref{mainthm}:
\begin{conjecture}
No such congruence as in Theorem \ref{mainthm} holds when $m_0\equiv 0\pmod{3}$.
\end{conjecture}

The reason for these phenomena is explained by our method of proof: the results of Theorem \ref{mainthm} follow from the existence of certain {Ramanujan-Kolberg identities} of the form
$$q\sum_{n=0}^\infty p_t(an+b)q^n \equiv \frac{1}{f_1^{at}} + \frac{1}{f_a^t},$$
which hold for specific triples $(a,b,t)$ when $a<24$, as proven in \cite{JudgeThesis,JKZ}.  Each case of such identities will yield the parity congruences stated above. Identities for $m_0=25$ and $m_0=29$ do exist but are more complex, and therefore produce a more complicated behavior. Finally, when $m_0\equiv 0\pmod{3}$, such identities are of a different nature entirely.

We also use this viewpoint to fill in an empty spot in \cite{KZ}: in that paper, we were unable to produce congruences for the odd-regular case $b_{11}$ beyond a finite family already known to Zhao, Jin, and Yao \cite{JYZ}, namely:
$$b_{11}(22n+B) \equiv 0 \pmod{2} \, \text{ for } \, B \in \{2,8,12,14,16\}.$$
With the techniques of this paper, we can now add the following:
\begin{theorem}\label{eleventhm}
It holds that
$$b_{11}(242n+B) \equiv 0 \pmod{2}$$
for all $n$, when $B \in \{ 28, 94, 182, 204, 226 \}$.
\end{theorem}

Let $\delta_t$ be the \emph{density} of the odd values of the $t$-multipartition function, i.e.,
$$\delta_t = \lim_{x \rightarrow \infty} \frac{\# \{n \leq x : p_t(n) {\ }\text{is odd} \}}{x},$$
if this limit exists.  Assuming suitable existence conditions,  in \cite{JKZ}, Theorem 2  we established that $\delta_t > 0$ implies $\delta_1 > 0$ for any $t \in \{5,7,11,13,17,19,23,25\}$. (This result was then fully generalized by the second author \cite{Zanello} to any odd value of $t$, $t\not\equiv 0\pmod{3}$; see  Theorem \ref{ZanDensities} below.)  By matching the densities of various subprogressions in $m$-regular partitions to those in $t$-multipartitions, Theorem \ref{mainthm} also provides a new set of implications concerning the parity of the partition function. We have:
\begin{corollary}\label{densitycorollary}
Let $m_0 - 2^j \in \{1, 5, 7, 11, 13, 17, 19, 23, 25\}$ for a congruence in Theorem \ref{mainthm}, and suppose the coefficients of any named subprogression have positive odd density in the relevant $m$-regular partitions. Then, assuming the odd density $\delta_1$ of $p(n)$ exists, it satisfies $\delta_1>0$.
\end{corollary}

If the odd density of the coefficients of a series is zero, the series is said to be \emph{lacunary mod 2}.  It was proven by Gordon and Ono \cite{GordonOno} (and generalized by Cotron, Michaelsen, Stamm, and Zhu \cite{CMSZ}; in what follows, CMSZ) that $f_m / f_1$ is lacunary mod 2 whenever $2^j > m_0$.

We now state the main conjecture of this paper. In particular, it fully characterizes the odd densities of the even-regular partition functions.
\begin{conjecture}\label{MainEvenConjecture}
Let $m$ be even, and write $m = 2^j m_0$ with $m_0$ odd.
\begin{enumerate}
\item If $2^j > m_0$, then in addition to the sequence $b_m(n)$ being lacunary mod 2 by \cite{GordonOno}, $b_m$ is identically even on infinitely many nonconstant, non-nested subprogressions $An+B$.
\item If $2^j < m_0$, then the odd density of $b_m$ is $1/2$. Further, the odd density of $b_m$ on \emph{any} nonconstant subprogression $An+B$ also equals $1/2$. (In particular, there exist no identically even nonconstant subprogressions in $b_m$.)
\end{enumerate}
\end{conjecture}

For instance, a sample statement proven in this paper that concerns the second half of clause (1) is: 
\begin{theorem}\label{m40even}
Both of $b_{40}(25n+9)$ and $b_{40}(25n+19)$ are identically zero mod 2.
\end{theorem}

The corresponding conjecture for multipartitions was stated by Judge and the present authors in \cite{JKZ}, and generalizes Parkin-Shanks' \cite{PaSh} well-known conjecture for $p(n)$:
\begin{conjecture}\label{multiconj} (\cite{JKZ}, Conjecture 1)
For any odd positive integer $t$, the density $\delta_t$ exists and equals $1/2$. Equivalently, if $t = 2^k t_0$ with $t_0\ge 1$ odd, then $\delta_t$ exists and equals $2^{-k-1}$.
\end{conjecture}

All conjectures given here, including Conjectures \ref{MainEvenConjecture} and \ref{multiconj}, are consistent with our \lq \lq master conjecture'' on the parity of eta-quotients presented in \cite{KZ}, which has motivated substantial portions of both works. We restate it here for context:
\begin{conjecture}\label{overallconj} (\cite{KZ}, Conjecture 4)
Let $F(q)=\sum_{n\ge 0} c(n)q^n$ be an eta-quotient, shifted by a suitable power of $q$ so powers are integral, and let $\delta_F$ be the odd density of the $c(n)$. Then:
\begin{enumerate}
\item[ i)] For any $F$, $\delta_F$ exists and satisfies $\delta_F\le 1/2$.
\item[ ii)] If $\delta_F= 1/2$, then for any nonnegative integer-valued polynomial $P$ of positive degree, the odd density of $c(P(n))$ is $1/2$. (In particular, for all nonconstant subprogressions $An+B$, $c(An+B)$ has odd density $1/2$.)
\item[ iii)] If $\delta_F< 1/2$, then the coefficients of $F$ are identically even on \emph{some} nonconstant subprogression.% (Note that this is not even \emph{a priori} obvious when $\delta_F=0$.)
\item[ iv)] If the coefficients of $F$ are not identically even on \emph{any} nonconstant subprogression, then they have odd density $1/2$ on \emph{every} nonconstant subprogression; in particular, $\delta_F= 1/2$.\\(Note that i), ii), and iii) together imply iv), and  iv) implies iii).)
\end{enumerate}
\end{conjecture}

\section{Background and Preliminary Notions}

A \emph{partition} of $n$ is defined as a weakly decreasing sequence $\lambda = (\lambda_1, \dots, \lambda_\ell)$ of positive integers that sums to $n$; we write $\vert \lambda \vert = n$ and denote their number by $p(n)$. An $m$-\emph{regular} partition of $n$ is then a partition in which none of the $\lambda_i$ is divisible by $m$; their number is $b_m(n)$.  A $t$-\emph{multipartition} of $n$ is an ordered tuple of partitions $(\alpha_1, \dots, \alpha_t)$ such that $\sum_{k=1}^t \vert \alpha_k \vert = n$; we denote their number by $p_t(n)$.

These sequences have generating functions, respectively:
\begin{align}
\sum_{n=0}^\infty p(n) q^n &= \frac{1}{f_1}; \\ \sum_{n=0}^\infty b_m(n) q^n &= \frac{f_m}{f_1}; \\ \sum_{n=0}^\infty p_t(n)q^n &= \frac{1}{f_1^t}.
\end{align}

We will make frequent use of the mod 2 reduction of Euler's Pentagonal Number Theorem,
$$f_1 \equiv \sum_{n \in \mathbb{Z}} q^{\frac{n}{2}(3n-1)},$$
and the fact that, again modulo 2, the cube of $f_1$ has nonzero coefficients precisely at the triangular numbers. That is,
$$f_1^3 \equiv \sum_{n=0}^\infty q^{\binom{n+1}{2}}.$$
These are well-known results in partition theory, and can be found in their unreduced form as identities (1.3.1) and (2.2.13) in \cite{Andr}, respectively.

We will frequently and without further comment employ the following trivial fact about powers of 2 for series:
\begin{lemma}\label{222}
For any power series $f(q) = \sum_{n=0}^\infty a(n) q^n$, it holds that
$$(f(q))^2 \equiv \sum_{n=0}^\infty a(n) q^{2n}.$$
In particular, $f_j^{2^d} \equiv f_{2^dj}$.
\end{lemma}

We will also make use of the  \emph{$U$ operator} (following \cite{OnoWeb}, equation (2.15)): given a power series $f(q) = \sum_{n=0}^\infty a(n) q^n$ and a positive integer $d$, define
$$f(q) \vert U(d) = \sum_{n=0}^\infty a(dn) q^{n}.$$
Then this operator has the multiplicative property that
$$\left( f(q) \vert U(d) \right) g(q) = \left( f(q) g\left(q^d\right) \right) \vert U(d).$$

As we mentioned earlier, the sequence $b_m(n)$ is lacunary mod 2 whenever $m = 2^j m_0$ with $m_0$ odd and $2^j > m_0$ (see \cite{GordonOno}). This is also a consequence of the fact that
$$\frac{f_{2^j m_0}}{f_1} \equiv \frac{f_{m_0}^{2^j}}{f_1},$$
along with the following more general theorem by CMSZ, which we rephrase here in a form different from but equivalent to that of their original paper for the case of lacunarity mod 2:
\begin{theorem}\label{cot} (\cite{CMSZ}, Theorem 1.1) Suppose that all of $\alpha_i, \gamma_i, r_i, s_i$ are positive integers, with the $\alpha_i$ and $\gamma_i$ all distinct. 
Let $F(q)=\frac{\prod_{i=1}^uf_{\alpha_i}^{r_i}}{\prod_{i=1}^tf_{\gamma_i}^{s_i}}$, and assume that
$$\sum_{i=1}^u \frac{r_i}{\alpha_i} \ge \sum_{i=1}^t s_i\gamma_i.$$
Then the coefficients of $F$ are lacunary mod 2.
\end{theorem}

If \emph{almost all} entries in a sequence are even (i.e., all entries except possibly a set of density zero), it is perhaps unsurprising if a subprogression within the sequence is identically even.  This means that in considering $b_m$ for $m$ even, the values of $m$ for which $2^j < m_0$ are often more interesting to study. We will mostly restrict our attention to such $m$ in this paper.

A \emph{Ramanujan-Kolberg identity} is an equation or congruence between a subprogression in a multipartition function and a linear combination of eta-quotients.  They were so named by Radu \cite{RaduRK}, after Ramanujan's ``most beautiful identity'' and Kolberg's further study \cite{Kolberg} of similar equations. For instance: 
\begin{align*}
\sum_{n=0}^\infty p(5n+4)q^n &= 5 \frac{f_5^5}{f_1^6}; \\ \sum_{n=0}^\infty p(7n+5)q^n &= 7 \frac{f_7^3}{f_1^4} + 49 q \frac{f_7^7}{f_1^8}.\end{align*}

In \cite{JudgeThesis,JKZ}, the present authors along with Judge established a large class of these identities mod 2.  The ones relevant for this paper are:
\begin{theorem}\label{RKidents} (\cite{JudgeThesis,JKZ})
The congruence
$$q \sum_{n=0}^\infty p\left(An+24^{-1}\right)q^n \equiv \frac{1}{f_1^A} + \frac{1}{f_A}$$
holds for $A \in \{5,7,11,13,17,19,23\}$, where $24^{-1}$ is taken modulo $A$.  Additionally,
$$q^2 \sum_{n=0}^\infty p(25n+24)q^n \equiv \frac{1}{f_1^{25}} + \frac{q}{f_1} + \frac{1}{f_5^5}$$
and
$$q^2 \sum_{n=0}^\infty p(29n+23)q^n \equiv \frac{1}{f_1^{29}} + \frac{q}{f_1^5} + \frac{1}{{f_{29}}}.$$
\end{theorem}

In \cite{JZ}, Judge and the second author extended the above and conjectured the existence of a broad, infinite family of Ramanujan-Kolberg identities:
\begin{conjecture}\label{JZconj} (\cite{JZ}, Conjecture 2.4) 
For odd integers $a$ and $t$, where $3 \vert t$ if $3 \vert a$, let $b \equiv \frac{t}{3} \cdot 8^{-1}$ if $3 \vert t$ and $b \equiv 24^{-1} \pmod{a}$ if not. Set $k = \left\lceil t(a^2-1)/(24a) \right\rceil$.  Then
$$q^k \sum_{n=0}^\infty p_t(an+b) q^n \equiv \sum_{d \vert a} \sum_{j=0}^{\lfloor k/d \rfloor} \frac{\epsilon^t_{a,d,j}q^{dj}}{f_d^{at/d-24j}},$$
for a suitable choice of the $\epsilon^t_{a,d,j} \in \{0,1\}$ such that $\epsilon^t_{a,1,0} = 1$ (i.e., the ``largest'' term appears) and $\epsilon^t_{a,d,j} = 0$ for $at/d-24j < 0$ (i.e., no negative powers appear).
\end{conjecture}

A primary interest of this conjecture is that the subprogression generating functions appearing on the left side would  lie in a particular nicely-described and relatively low-dimensional subspace of modular forms.  (We will not require formal definitions for work with modular forms in this paper.)  Chen \cite{Chen} recently proved Conjecture \ref{JZconj} for (among other cases) all $a\ge 3$ prime, and using his result, the second author \cite{Zanello} was able to establish the density implications conjectured in \cite{JudgeThesis} in full generality. We have:
\begin{theorem}\label{ZanDensities} (\cite{Zanello}, Theorem 4)
\begin{enumerate}
\item If there exists an integer $t \equiv \pm 1 \pmod{6}$ such that $\delta_t > 0$, and all densities $\delta_i$ exist for $i \leq t$, $i \equiv \pm 1 \pmod{6}$, then $\delta_1 > 0$.
\item If there exists an integer $t \equiv 3 \pmod{6}$ such that $\delta_t > 0$, and all densities $\delta_i$ exist for $i \leq t$, $i \equiv 3 \pmod{6}$, then $\delta_3 > 0$.
\end{enumerate}
\end{theorem}

\begin{remark} In clause (1) of the above theorem we note that the two instances of $\pm$ are independent; e.g., starting from $\delta_{11}$ we would require all of $\delta_7$, $\delta_5$, and $\delta_1$ to exist.  This requirement is stated to ensure that the theorem holds when the intervening Ramanujan-Kolberg identities are not known; it can be reduced in specific cases.  For instance, as pointed out earlier, if $\delta_1$ exists, then $\delta_{11} > 0$ implies $\delta_1 > 0$ directly.
\end{remark}

\section{Proofs}

\subsection{The proof of Theorem \ref{mainthm}}

We begin by showing our first main theorem.
\begin{proof}
The proofs of all claims of Theorem \ref{mainthm} can be obtained by a suitable analysis of the Ramanujan-Kolberg identities in Theorem \ref{RKidents}.  The logic for each is similar, so we illustrate it here with a representative case, the 104-regular partitions.

Since $104 = 13 \cdot 2^3$, we begin with the $A=13$ case of Theorem \ref{RKidents}, namely the Ramanujan-Kolberg identity
$$q \sum_{n=0}^\infty p(13n+6)q^n \equiv \frac{1}{f_1^{13}} + \frac{1}{f_{13}}.$$

This can be reformulated as:
$$\frac{q^7}{f_1} \vert U(13) \equiv \frac{1}{f_1^{13}} + \frac{1}{f_{13}}.$$

Multiplying through by $f_1^8$, recalling that $f_1^8 \equiv f_8$, and using the multiplicative property of $U$ to bring the factor inside the operator, we obtain:
$$q^7 \frac{f_{104}}{f_1} \vert U(13) \equiv \frac{1}{f_1^5} + \frac{f_8}{f_{13}},$$
or equivalently,
$$\sum_{n=0}^\infty b_{104} (13n-7) q^{n} \equiv \frac{1}{f_1^5} + \frac{f_8}{f_{13}}.$$

Now observe that, if we set
$$\frac{f_8}{f_{13}} = \sum_{n=0}^\infty a(n) q^n,$$
the coefficients $a(n)$ can be written using the Pentagonal Number Theorem as a recurrence in the coefficients of $f_{13}$, with terms given by $f_8$. These latter are 8 times the pentagonal numbers, i.e., $4k(3k+1)$ for $k \in \mathbb{Z}$.  But by completing the square, we have that
$$4k(3k+1) =12k^2 +4k \equiv 12\left(k+6^{-1}\right)^2 - 3^{-1} \pmod{13},$$
where inverses are taken mod 13. Hence, mod 13, the odd coefficients in $f_8$ must appear in degrees
$$12x^2 - 3^{-1} \equiv 12x^2 + 4 \pmod{13}.$$

As $x$ ranges through the integers, $x^2$ ranges through the quadratic residues mod 13. Thus, it is easy to see that $12x^2+4 \in \{ 0, 1, 3, 4, 5, 7, 8\} \pmod{13}$.

We deduce that $f_8/f_{13}$ has no odd coefficients for $q^n$ if $n \equiv 2, 6, 9, 10, 11,12 \pmod{13}$. Therefore, if $n \equiv 2, 6, 9, 10, 11, 12 \pmod{13}$, then
$$b_{104}(13n-7) \equiv p_5(n),$$
as desired.

(An example of this congruence is that $b_{104}(169n+19) \equiv p_5(13n+2).$)

All other claims of Theorem \ref{mainthm} are proved with exactly analogous logic. Begin with the clause of Theorem \ref{RKidents} where $A = m_0$, and multiply through by $f_1^{2^j}$. One term on the right will be the $t$-multipartition function with $t = m_0 - 2^j$, while the other will miss just under half the residues modulo $m_0$, namely those avoided by $2^j$ times a pentagonal number. %\hfill $\Box$
\end{proof}

\subsection{The proof of Theorem \ref{m40even}.}

According to Conjecture \ref{MainEvenConjecture} above, for $m = 2^jm_0 $ with $2^j > m_0$, the sequences $b_m(n)$, which are known to be lacunary mod 2 by \cite{GordonOno}, contain infinitely many even  subprogressions.  Since even subprogressions are usually easy to establish, we do not occupy ourselves here with more than a single example.  Many other candidates present themselves to a short calculation.

We now prove Theorem \ref{m40even}.
\begin{proof}
Considering the case $A=5$ of Theorem \ref{RKidents}, we have:
$$\frac{q f_{40}}{f_1} \vert U(5) \equiv f_1^3 + \frac{f_8}{f_5}.$$

The first term on the right side is odd at $q^n$ for $n$ a triangular number, while in the second summand, $f_8$ contributes terms at 8 times the pentagonal numbers, which then recur with a denominator strictly in powers of $q^5$.  It is easy to see that triangular numbers and numbers that are 8 times a pentagonal number are both always 0, 1, or 3 mod 5, and thus miss the subprogressions 2 and 4 mod 5.  It follows that
$$b_{40}(5(5n+\{2,4\}) - 1)$$
must be even, and the theorem is proved. %\hfill $\Box$
\end{proof}

Note that the condition $2^d > m_0$ makes all terms on the right side of the relevant Ramanujan-Kolberg identity in Conjecture \ref{JZconj} lacunary by CMSZ, after we multiply through by $f_1^{2^d}$. Therefore, since series that are lacunary mod 2 are expected to be rich in even subprogressions, similar analysis is likely to establish many other theorems along the same lines.

\subsection{The proof of Theorem \ref{largermod}.}

We first need the main result of a recent paper by Cherubini and Mercuri \cite{CM}, which nicely characterized the parity of $b_8$. 
\begin{theorem}\label{8reg} (\cite{CM}, Theorem 1.1)
For $n \geq 0$, $b_8(n)$ is odd if and only if
$$24n+7 = p^{4a+1}c^2,$$
for some prime $p \nmid c$ and $a \in \mathbb{N}$.\\
(Note that, since ${f_8}/{f_1} \equiv f_1^7$, the same statement holds for the coefficients of $f_1^7$.)
\end{theorem}

We only remark here that the original proof given in \cite{CM} made a substantial use of modular forms. At the end of this section, we provide a much shorter, elementary argument.

We are now ready for the proof of Theorem \ref{largermod}.
\begin{proof}
For the $m_0 = 25$ case of the theorem, we recall the modulus 25 clause of Theorem \ref{RKidents}:
$$q^2 \sum_{n=0}^\infty p(25n+24)q^n \equiv \frac{1}{f_1^{25}} + \frac{q}{f_1} + \frac{1}{f_5^5}.$$

In terms of the $U$ operator, we have:
$$q^{26} \frac{1}{f_1} \vert U(25) \equiv \frac{1}{f_1^{25}} + \frac{q}{f_1} + \frac{1}{f_5^5}.$$

For the 400-regular partitions, we multiply through by $f_{16}$ to obtain:
$$q^{26} \frac{f_{400}}{f_1} \vert U(25) \equiv \frac{1}{f_1^9} + q f_1^{15} + \frac{f_{16}}{f_5^5}.$$

We now observe that odd coefficients in $f_{16}/f_5^5$ appear on $q^n$ only if $n \equiv 0, 1, 2 \pmod{5}$.  Noting that
$$q f_1^{15}\equiv qf_1^3 \cdot f_4^3,$$
we easily see that the odd coefficients of $q f_1^{15}$ appear only if the exponent can be written as
$$1 + \binom{m+1}{2} + 4 \binom{n+1}{2},$$
for some $m,n \in \mathbb{N}$.  By an analysis similar to the above, we see that such numbers cannot be 3, 9, 18, or 39 mod 45. Hence, those exponents are missed by both of the latter terms, and the claim is proved.

The $b_{50}$ and $b_{100}$ cases of the theorem are analyzed similarly but only require considering $q f_1$ and $q f_1^3$ in the middle terms, respectively.

For $m_0 = 29$, we use the third clause of Theorem \ref{RKidents}, and obtain:
$$q^{35} \frac{f_{232}}{f_1} \vert U(29) \equiv \frac{1}{f_1^{21}} + q f_1^3 + \frac{f_8}{f_{29}}.$$
The rest of the argument is similar and will be omitted. 

We need an additional ingredient for the proof of the $b_{200}$ clause.  First, note that we want to analyze the congruence:
\begin{equation}\label{200eq}
q^{26} \frac{f_{200}}{f_1} \vert U(25) \equiv \frac{1}{f_1^{17}} + q f_1^7 + \frac{f_8}{f_5^5}.
\end{equation}

Now observe that if $n \equiv \{ 21, 28, 42 \} \pmod{49}$, then $24n+7 = 7 K$ with $7 \nmid K$, which would require $p=7$ in Cherubini-Mercuri's Theorem \ref{8reg}. Also, $K \equiv \{ 3,6,5 \} \pmod{7}$, respectively.

But the above are exactly the quadratic nonresidues modulo 7, and so they cannot equal $m^2$ for any integer $m$.  Therefore, $q f_1^7$ has no odd coefficients in the subprogressions $49n+ \{ 23, 29, 43 \} $.

Further, $f_8/f_5^5$ avoids subprogressions 2 and 4 mod 5, and the subprogressions stated in the theorem are those solving the Chinese remainder problem modulo $245 = 5 \cdot 49$, with the six possible pairs of residues.  Since both the second and third summands on the right side of (\ref{200eq}) avoid these subprogressions, we have that the left side is congruent to $1/f_1^{17}$ in said subprogressions. That is,
$$b_{200}(25(245n+B) - 26) \equiv p_{17} (245n+B),$$
for all $n$ and the given values of $B$. This completes the proof of the theorem.
\end{proof}

\begin{remark}
Note that, as $m_0$ grows, the relevant Ramanujan-Kolberg identities tend to accumulate more terms, and the analysis likely becomes more difficult in multiple ways.

For instance, when $m_0 = 29$, it is possible that a similar behavior holds for some large subprogressions in the case $b_{464}$, where the middle term is $q f_1^{11}$. However, we suspect greater obstacles to the existence of these behaviors for $b_{58}$ and $b_{116}$, where the middle terms are $q/f_1^3$ and $q/f_1$, respectively.
\end{remark}

It follows from a classical result of Landau \cite{Landau} on quadratic representations that $q f_1^7\equiv qf_1^3 \cdot f_4$ is lacunary mod 2. Hence, we have the following statement about parities \emph{almost always} matching (i.e., matching outside of a set of density zero).

\begin{theorem}\label{200density}
The values $b_{200}(25(5n+2)-26)$ and $p_{17}(5n+2)$ almost always match in parity, as do $b_{200}(25(5n+4)-26)$ and $p_{17}(5n+4)$.
\end{theorem}

We note that similar theorems yielding ``almost always congruences'' can be established for $b_m$ whenever $m$ is such that all terms in the relevant Ramanujan-Kolberg identity after multiplication by $f_1^{2^j}$ are lacunary, except for the leading term $1/f_1^{m_0-2^j}$ and the final term $f_1^{2^j}/f_{m_0}$ (regardless of the values of the $\epsilon_{a,d,j}^t$).

One curious example of such a theorem is the following:
\begin{theorem}\label{65537thm}
For any $B$ not congruent to -1 times a pentagonal number modulo 65537, the congruence
$$b_{4294377472}\left(65537^2 n + 65537 B - 2731\right) \equiv p(65537n+B)$$
holds for almost all $n$. 
\end{theorem}

\begin{proof}
Here we are considering
$$m = 4294377472 = 65537 \cdot 2^{16} = \left(2^{16} + 1\right) \cdot 2^{16},$$
where 65537 is prime (it is currently the largest known Fermat prime).  Hence, in the general Ramanujan-Kolberg identity established by Chen  \cite{Chen}, we have $k = \left\lceil \frac{65537^2 - 1}{24 \cdot 65537} \right\rceil = 2731$.

It follows that for a suitable choice of the $\epsilon_i \in \{0,1\}$,
$$q^{2731} \frac{f_{65537}^{2^{16}}}{f_1} \vert U(65537) \equiv \frac{1}{f_1} + \frac{f_{65536}}{f_{65537}} + \epsilon_1 q f_1^{23} + \epsilon_2 q^2 f_1^{47} + \dots + \epsilon_{2730} q^{2730} f_1^{65519}.$$

The summands after the second on the right side are all lacunary mod 2, for instance by CMSZ. Therefore, since any finite sum of lacunary series is lacunary, the congruence
$$q^{2731} \frac{f_{65537}^{2^{16}}}{f_1} \vert U(65537) \equiv \frac{1}{f_1} + \frac{f_{65536}}{f_{65537}}$$
is almost always true, and the claim follows.
\end{proof}

\begin{remark}
We have some confidence that this is the first theorem in the literature explicitly stated about 4294377472-regular partitions!
\end{remark}

\subsection{The proof of Theorem \ref{eleventhm}}

\begin{proof}
To show the theorem, we consider the clause $A=11$ of Theorem \ref{RKidents}.  Stated in terms of the $U$ operator, after multiplying through by $f_1$, we have:
$$q^5 \frac{f_{11}}{f_1} \vert U(11) \equiv \frac{1}{f_1^{10}} + \frac{f_1}{f_{11}}.$$

Clearly, $1/f_1^{10}$  has no terms with odd exponent $q^{2n+1}$. We then calculate that the pentagonal numbers miss the subprogressions 3, 6, 8, 9, and 10 mod 11, which implies that the right side of the last displayed congruence has no odd coefficients in the subprogressions $22n+\{ 3,9,17,19,21 \} $.  Thus,
$$q^5 \frac{f_{11}}{f_1} \vert U(11)$$
avoids said subprogressions, i.e., the coefficient of $q^k$ is even whenever $k$ is of the form
$$11(22n+\{ 3,9,17,19,21\} ) - 5 = 242n + \{ 28, 94, 182, 204, 226\} .$$
This proves the theorem. %\hfill $\Box$
\end{proof}

\subsection{New proof of Cherubini-Mercuri's main result.}

In the final part of this section, we provide a shorter and  elementary proof of Cherubini-Mercuri's main theorem from \cite{CM} (see Theorem \ref{8reg} above). In particular, we avoid the modular-form machinery employed in \cite{CM}, and solely rely on a classical fact on integer representations by sums of squares.
\begin{proof}
By completing the square, one computes that the nonzero coefficients in
$$f_1^7 \equiv f_1^4 \cdot f_1^3 \equiv \left( \sum_{a \in \mathbb{Z}} q^{2a(3a-1)} \right) \left( \sum_{b \in \mathbb{Z}} q^{b(2b-1)} \right)$$
must appear on terms $q^n$ such that
$$24n+7 = 4(6a-1)^2 + 3(4b-1)^2 = (12a-2)^2 + 3(4b-1)^2.$$
Simple arithmetic considerations give us that whenever
$$24n+7 =x^2 + 3y^2$$
for $x, y \in \mathbb{Z}$, $y$ must necessarily be odd, $x \equiv 2 \pmod{4}$, and $x \not\equiv 0 \pmod{3}$. Thus, equivalently, we can write $x=\pm (12a-2)$ and $y=\pm (4b-1)$, for $a,b\in \mathbb{Z}$.

It easily follows from the above that the coefficient of $q^n$ in $f_1^7$ is of parity precisely $1/4$ the number, $r_3(s)$, of representations of $s=24n+7$ as $x^2 + 3y^2$, for $x, y \in \mathbb{Z}$.

Now note that $s\equiv 1 \pmod{6}$, and consider its prime factorization
$$s = p_1^{a_1} \cdots p_k^{a_k} q_1^{b_1} \cdots q_\ell^{b_\ell},$$
where the $p_i \equiv 1 \pmod{3}$ and the $q_j \equiv 2 \pmod{3}$.  It is a classical fact (see, e.g., \cite{lor} or the more recent \cite{BL22}) that if $s$ is representable as $x^2 + 3y^2$ for $x, y \in \mathbb{Z}$, then all of the $b_j$ are even and
$$r_3(s) = 2(a_1+1) \cdots (a_k+1).$$
Since
$$\frac{1}{4} \cdot 2(a_1+1) \cdots (a_k+1)$$
is odd precisely when all of the $a_i$ are even except for exactly one, which must be 1 mod 4, the proof of the theorem is complete.
\end{proof}

\section{Future Research Directions}

We conclude by presenting a brief selection of other problems, on top of the conjectures stated earlier, that naturally arise from this work. First note that, to the best of our knowledge, no theorem currently exists that \emph{simultaneously} states properties of $m$-regular partitions such as those in this paper for an \emph{infinite} set of values of $m$. In this direction, we suggest the following two open-ended lines of investigations:
\begin{problem}
Give a general theorem showing the existence of infinitely many similarities of the type conjectured in the first section, preferably one that effectively identifies such similarities.
\end{problem}

Let as usual $m = 2^j m_0$, and assume $2^j > m_0$. Recall that $b_m$ is lacunary mod 2 by \cite{GordonOno}, and that Conjecture \ref{MainEvenConjecture} predicts the existence of infinitely many subprogressions where $b_m$ is identically even. We therefore ask:
\begin{problem}
Can a formula or concise algorithm be given that takes such values of $m$ as input and produces at least one output pair $(A,B)$ with $A>0$ and $b_m(An+B)$ identically even?
\end{problem}

In a different direction, we also ask:
\begin{problem}
Identify cases of the general Ramanujan-Kolberg identity in which ``many'' of the $\epsilon^t_{a,d,j}$ are 0; this will yield potentially more manageable congruences with fewer terms.
\end{problem}

Further, while Conjecture \ref{MainEvenConjecture} offers a complete description of the odd densities of the even-regular partitions, so far we do not have a similar conjecture to control the density behavior for arbitrary $m$ odd. In fact, the latter densities appear to vary more.

For instance, Judge and these two authors proved in \cite{JKZ}, Theorem 7 that  (assuming existence) the odd density of $b_5$ is $1/4$ that of $b_{20}$, and similarly, the odd density of $b_7$ is $1/2$ that of $b_{28}$. Thus, under Conjecture \ref{MainEvenConjecture}, the odd density of $b_5$ equals $1/8$, and that of $b_7$ is $1/4$. Extensive computational evidence suggests that also for other odd values of $m$, these odd densities are likely to be strictly contained between $0$ and $1/2$.
\begin{problem}
Characterize, even conjecturally, the values of the odd densities of $b_m$ for $m$ odd, in analogy to Conjecture \ref{MainEvenConjecture}.
\end{problem}

Finally, it is interesting to note that, in all instances where the odd density of an eta-quotient is currently known to exist, that density is equal to zero (see \cite{KZ}, and also Conjecture \ref{overallconj} above). It would be a significant (and consequential!) result to show the \emph{existence} of a {positive} odd density, for any eta-quotient. We explicitly state it as a problem here for the even-regular partitions studied in this paper, as follows:
\begin{problem}
Show that the odd density of $b_m$ exists, for any even value $m = 2^j m_0$ such that $2^j < m_0$.
\end{problem}

\section{Acknowledgements}
We thank the referee for a careful reading of our manuscript and helpful comments. The second author was partially supported by a Simons Foundation grant (\#630401).

\end{document}